\documentclass[11pt]{amsart}

\usepackage{graphicx}
\usepackage[labelsep=space]{caption}
\usepackage{amsfonts}
\usepackage{tikz-cd}
\usepackage{amsthm}
\usepackage{nccmath}
\usepackage{etex}
\usepackage{amssymb,mathtools}
\usepackage{mathrsfs}
\usepackage{dsfont}
\usepackage[all]{xy}
\usepackage{microtype}
\newcommand{\widesim}[2][1.5]{
	\mathrel{\overset{#2}{\scalebox{#1}[1]{$\sim$}}}
}

\usepackage{enumerate}
\usepackage[utf8]{inputenc}
\usepackage[T1]{fontenc}
\usepackage{lmodern} 
\usepackage{latexsym}
\usepackage{MnSymbol}
\usepackage{nicefrac}
\usepackage{microtype}
\usepackage{color}
\usepackage{tikz-cd}
\usepackage{blindtext}

\usepackage{old-arrows}
\usepackage{empheq}
\usepackage{extpfeil}
\usepackage{hyperref}
\usepackage{theoremref}

\usepackage{float}
\makeatletter
\renewenvironment{proof}[1][\proofname]{%
	\par\pushQED{\qed}\normalfont%
	\topsep6\p@\@plus6\p@\relax
	\trivlist\item[\hskip\labelsep\bfseries#1\@addpunct{.}]%
	\ignorespaces
}{%
	\popQED\endtrivlist\@endpefalse
}
\makeatother

\usepackage{geometry}
\usepackage{secdot}
\usetikzlibrary{decorations.markings}
\tikzset{double line with arrow/.style args={#1,#2}{decorate,decoration={markings,%
			mark=at position 0 with {\coordinate (ta-base-1) at (0,1pt);
				\coordinate (ta-base-2) at (0,-1pt);},
			mark=at position 1 with {\d[#1] (ta-base-1) -- (0,1pt);
				\d[#2] (ta-base-2) -- (0,-1pt);
}}}}
\usepackage{ifxetex}
\usepackage{etoolbox}
\ifxetex
\usepackage{unicode-math}
\makeatletter
\patchcmd{\arrowfill@}{-7mu}{-14mu}{}{}
\patchcmd{\arrowfill@}{-7mu}{-14mu}{}{}
\patchcmd{\arrowfill@}{-2mu}{-4mu}{}{}
\patchcmd{\arrowfill@}{-2mu}{-4mu}{}{}
\makeatother
\fi

\theoremstyle{plain}
\newtheorem{theorem}{Theorem}[section]

\newtheorem{corollary}[theorem]{Corollary}
\newtheorem{proposition}[theorem]{Proposition}
\theoremstyle{remark}
\newtheorem{remark}[theorem]{Remark}
\theoremstyle{definition}
\newtheorem{definition}[theorem]{Definition}

\begin{document}

\title{Strong Nielsen equivalence on the punctured disc}
\author{Stavroula Makri}
\address{Faculty of Mathematics and Computer Science, Nicolaus Copernicus University, ul. Chopina 12/18, 87-100 Toruń, Poland}
\email{makri@mat.umk.pl}

\begin{abstract}
Let $f$ be an orientation-preserving homeomorphism of the 2-disc $\mathbb{D}^2$ that fixes the boundary pointwise and leaves invariant a finite subset in the interior of $\mathbb{D}^2$.
We study the strong Nielsen equivalence of periodic points of such homeomorphisms $f$ and we give a necessary and sufficient condition for two periodic points to be strong Nielsen equivalent in the context of braid theory. In addition, we present an application of our result to the trace formula given by Jiang--Zheng, deducing that the obtained forced periodic orbits belong to different strong Nielsen classes.
\end{abstract}

\maketitle

\section{Introduction}
A classical problem related to Nielsen theory deals with the question of determining the minimum number of fixed points among all maps homotopic to a given map $f$ from a compact space to itself.   
In particular, let $X$ be a compact connected polyhedron, $f:X\to X$ a continuous self-map, and let $Fix(f)=\{x\in X\ |\ f(x)=x\}$ be the set of fixed points of $f$. One is interested in studying the so-called minimal number of the map $f$ denoted by $MF[f]$ and defined as $MF[f]=\text{min}\{\#Fix(g)\ | \ g\sim f\}$, where $\sim$ means homotopic. We define an equivalence relation on the set of fixed points of $f$ in the following way. Two fixed points belong in the same fixed point class if there exists a path $\gamma$ joining them, such that $f(\gamma)\sim \gamma$ keeping the endpoints fixed during the homotopy. The equivalence classes with respect to this relation are called the Nielsen classes of $f$. The Nielsen number $N(f)$ is defined as the number of Nielsen classes having non-zero fixed-point index sum. The number $N(f)$ is a homotopy invariant which gives a lower bound for $MF[f]$, that is $N(f)\leq MF[f]$.

To compute $N(f)$ one needs to decide whether two fixed points belong to the same fixed point class. This geometric problem was described in an algebraic context by Reidemeister in the following way. One can associate a Reidemeister class to each fixed point of $f$. Then two fixed points belong to the same fixed point class if and only if their Reidemeister classes coincide. Let $\phi:G\to G$ be an endomorphism of a group $G$. Then $u,v\in G$ are said to be Reidemeister equivalent or $\phi$-conjugate if there exists $w\in G$ such that $v=\phi(w)\cdot u\cdot w^{-1}$. This equivalence relation gives rise to the Reidemeister classes. Briefly, one interprets $\phi$ as the endomorphism $f_{\pi}$ induced by $f$ on the fundamental group $\pi_1(X)$ of $X$. 
		
In \cite{guaschi}, Guaschi uses braid group theory to give a necessary and sufficient condition, in terms of a conjugacy problem in the braid group, for distinguishing Reidemeister classes for automorphisms of the free group $F_n$ of rank $n$. These automorphisms are induced by orientation-preserving homeomorphisms of the 2-disc $\mathbb{D}^2$, relative to an $n$-point invariant set $A$, lying in the interior of $\mathbb{D}^2$. We will briefly give further insights of this result. 
 Due to Alexander's trick we know that any orientation preserving homeomorphism of $\mathbb{D}^2$ that fixes the boundary pointwise is isotopic to the identity map.
Let $B_n$ denote the Artin braid group on $n$-strands, $B_{n,1}$ the subgroup of $B_{n+1}$ whose induced permutation stabilizes $(n+1)$ and let $U_{n+1}$ be the kernel of the homomorphism $B_{n,1}\to B_n$, defined geometrically by removing the $(n+1)^{st}$ strand. Given an orientation-preserving homeomorphism $f$ of $\mathbb{D}^2$ that leaves invariant an $n$-point set $A$ and fixes the boundary pointwise, one can associate to $A$ a braid $\beta\in B_n$ by fixing an isotopy between the identity and $f$. The strands of $\beta$ appear naturally by following $A$ under the isotopy. Similarly, every fixed point $x\in \text{Int}(\mathbb{D}^2)\setminus A$ defines a braid $\beta_x$ in $B_{n,1}\subset B_{n+1}$, by following $A\cup\{x\}$ under the isotopy, whose induced permutation stabilizes $(n+1)$. Additionally, to any fixed point $x\in \text{Int}(\mathbb{D}^2)\setminus A$ one can associate an element $u\in\mathbb{F}_n$, since $U_{n+1}\cong \mathbb{F}_n$. In Section \ref{s1} and Section \ref{s2}, we present in detail how periodic points, and thus also fixed ones, are related to braid elements.
 In Theorem 1 of \cite{guaschi}, stated in Section \ref{s3}, it is shown that, given $\phi$ an automorphism of $F_n$, induced by $\beta\in B_n$ (associated to the invariant $n$-point set $A$), two elements $u,v\in \mathbb{F}_n$ belong in the same Reidemeister class if and only if the associated braids $\beta_u, \beta_v$ are conjugate in $B_{n,1}$ via an element of $U_{n+1}$.
 Consequently, any braid conjugacy invariant, such as linking number properties, Thurston type, application of Garside's algorithm and link invariants may be used to show that $\beta_u$ and $\beta_v$ are not conjugate in $B_{n+1}$ and thus that the two elements $u,v\in \mathbb{F}_n$ are not Reidemeister equivalent.
 In Section \ref{s3},  we show that an equivalent way of interpreting this result could be the following: the fixed points $x,y$, that correspond to $u, v$ respectively, are in the same Nielsen class, while taking into consideration the invariant $n$-point set $A$, if and only if $\beta_u, \beta_v$ are conjugate in $B_{n,1}$ via an element of $U_{n+1}$.

This work is inspired by \cite{guaschi} and aims to study the behaviour of periodic points with respect to a given invariant $n$-point set of orientation-preserving homeomorphisms of $\mathbb{D}^2$. In particular, in our first result we provide a criterion to decide when two periodic points of an orientation-preserving homeomorphism of a compact, connected, orientable surface $M$ are strong Nielsen equivalent in the sense of Asimov--Franks \cite{asimov} by further assuming contractible isotopy. More precisely, let $x,y$ be two periodic points of the same period of a homeomorphism $f$. We say that $x,y$ are strong Nielsen equivalent if there exists a contractible isotopy $f_t:f\simeq f$ and a path $\gamma:[0,1]\to M$ such that $\gamma(0)=x$, $\gamma(1)=y$, and $\gamma(t)$ is a periodic point, for every $t\in (0,1)$, of the same period as $x$ and $y$. An isotopy is said to be a deformation of another isotopy if the corresponding paths in the group of homeomorphisms are homotopic with fixed endpoints. A self-isotopy is called contractible if it is a deformation of the trivial isotopy. We show that two periodic points are strong Nielsen equivalent if and only if they correspond to the same class inside the corresponding mapping class group. Such criterion in the case of $\mathbb{D}^2$ is given in terms of a conjugacy problem in a braid group, as we will see in Section \ref{s2}.   

Let us fix some notation before stating the result. Given a compact, connected orientable surface $M$ we consider orientation-preserving homeomorphisms $f:M\to M$ that are isotopic to the identity. We denote by $P_n(f)$ the set of all periodic points with period $n$ of $f$ and by $o(x,f)$ the orbit of a periodic point $x$ of $f$.  For $n\in \mathbb{N}$, let $X_n\in \text{Int}(M)$ be a set of $n$ distinct points. Given a periodic orbit, $o(x,f)$, of period $n$, choose a homeomorphism $h:M\to M$ isotopic to the identity, such that $h(o(x,f))=X_n$. Then, let the strong Nielsen type of the orbit, denoted by $snt(x,f)$, be the conjugacy class of $[hfh^{-1}]$ in MCG$(M, X_n)$, where $[\cdot]$ represents the isotopy class and MCG$(M, X_n)$ the mapping class group of $M$ relative to $X_n$.  

For the following proposition we assume that $M$ is either a compact, connected orientable surface with negative Euler characteristic or the 2-disc $\mathbb{D}^2$, where in this case we restrict our attention to homeomorphisms that fix pointwise the boundary; see Remark \ref{extra}.

\begin{proposition}\label{thm1}
	Let $f:M\to M$ be an orientation-preserving homeomorphism and let $x,y\in P_n(f)$. It holds that $x,y$ are strong Nielsen equivalent, denoted by $x\widesim{\text{\tiny SN}}y$, if and only if $snt(x,f)=snt(y,f)$.
\end{proposition}

A direct consequence of Proposition \ref{thm1} is the following result.

\begin{corollary}\label{corollary}
	Let $f:\mathbb{D}^2\to\mathbb{D}^2$ be an orientation-preserving homeomorphism that fixes the boundary pointwise and let $x,y\in\text{Int}(\mathbb{D}^2)$ such that $x,y\in P_m(\mathbb{D}^2)$. Then $x\widesim{\text{\tiny SN}}y$ if and only if, for some $c\in B_m$, $\beta_{o_x}=c\cdot\beta_{o_y}\cdot c^{-1}$, where $\beta_{o_x}, \beta_{o_y}\in B_m$.
\end{corollary} 
Here, $\beta_{o_x}$ and $\beta_{o_y}$ are the $m$-braids in $\mathbb{D}^2\times[0,1]$ associated to $o(x,f)$ and $o(y,f)$ respectively, for a fixed isotopy $f_t:f\simeq id_{\mathbb{D}^2}$, as we will see in Section \ref{s2}.

Note that one should expect this result of Proposition \ref{thm1} as it has been stated in \cite{boy} by Boyland that one can show this equivalence between strong Nielsen classes and strong Nielsen types.

As we will see in Section \ref{s1}, there is also a definition that gives the strong Nielsen equivalence on periodic orbits uses the suspension flow, as described below. We recall that the mapping torus $M_f$ of a homeomorphism $f:M\to M $ is the quotient space ${M\times\mathbb{R}}/{\sim}$, with $(x,s+1)\sim (f(x),s)$, for every $x\in M$ and $s\in \mathbb{R}$.
To obtain the suspension flow $\psi_t$ on $M_f$, one takes the unit speed flow in the $\mathbb{R}$ direction on $M\times\mathbb{R}$ and then projects it to $M_f$. Let $p:M\times[0,1]\to M_f$ be the projection and $x\in M$, then $\gamma_x$ denotes the orbit of $p(x,0)$ under $\psi_t$. When $x$ is a periodic point, then $\gamma_x$ can be viewed as a simple closed curve in $M_f$. For $x,y\in P_n(f)$, the periodic orbits $o(x,f)$ and $o(y,f)$ are strong Nielsen equivalent, denoted by $o(x,f)\widesim{\text{\tiny SN}}o(y,f)$, if $\gamma_x$ is freely isotopic to $\gamma_y$ in $M_f$.
For the considered cases of $M$, Proposition \ref{thm1} is equivalent to the one where we replace periodic points $x\widesim{\text{\tiny SN}}y$ with periodic orbits $o(x,f)\widesim{\text{\tiny SN}}o(y,f)$, because in these cases all self-isotopies are contractible, see Remarks \ref{equiv} and \ref{extra}. Note that one can prove this result without asking for contractible self-isotopy. However, the absence of contractible-isotopy does not allow us to have for free the equivalent result for periodic orbit. Thus, for a more general $M$, one needs to also prove that the self-isotopy is contractible, in order to obtain the equivalent result for periodic orbits.  

It is important to notice that the strong Nielsen equivalence of periodic orbits of an orientation-preserving homeomorphism $f$ does not take into consideration their interaction with another possible invariant set $A$ of $f$. That is, two periodic orbits may be strong Nielsen equivalent, but when we take into consideration a given invariant set $A$ of $f$, then they may not be strong Nielsen equivalent. For this reason, we introduce and extend the definition of strong Nielsen equivalence of two periodic orbits to strong Nielsen equivalence, with respect to $A$, as follows: Given a compact, connected orientable surface $M$ we consider orientation-preserving homeomorphisms $f:M\to M$. 
Let $A\subset \text{Int}(M)$ an invariant $n$-point set of $f$ and $x,y\in P_m(f)$, such that $x\nin A$ and $y\nin A$. We say that the periodic orbits of $x$ and $y$ are strong Nielsen equivalent, with respect to $A$, denoted by $o(x,f)\widesim{\text{\tiny SN}}_A o(y,f)$, if $\gamma_x$ is freely isotopic to $\gamma_y$, while taking into consideration and keeping fixed the orbits $\gamma_{a_j}$ of $p(a_j,0)$ in $M_f$, for every $a_j\in A$. That is, the base points of $\gamma_x$ and $\gamma_y$ are allowed to vary through the isotopy, while the closed orbits of the points of $A$ are fixed. Note that it is possible to have $\gamma_{a_j}=\gamma_{a_k}$, for $i\neq j$.

For our second main result, we focus in the case where $M=\mathbb{D}^2$ and we give a necessary a sufficient condition for two periodic orbits to be strong Nielsen equivalent, with respect to a given invariant $n$-point set $A$ of $f$. 
More precisely, let $f:(\mathbb{D}^2, A; \partial \mathbb{D}^2)\to (\mathbb{D}^2, A; \partial \mathbb{D}^2)$ be an orientation-preserving homeomorphism of $\mathbb{D}^2$ that fixes the boundary pointwise and leaves invariant an $n$-point set $A\subset \text{Int}(\mathbb{D}^2)$. We prove that the problem of distinguishing strong Nielsen classes of periodic orbits, with respect to an invariant set $A$, is equivalent to studying the conjugacy problem in a braid group. Let $F_n(\mathbb{D}^2)$ be the $n^{th}$ configuration space of $\mathbb{D}^2$. We recall that one way of defining the braid group $B_n$ is as the fundamental group of the quotient of $F_n(\mathbb{D}^2)$ by the symmetric group of degree $n$, $S_n$. That is, $B_n=\pi_1(F_n(\Sigma)/S_n)$. Similarly, $B_{n,m}=\pi_1(F_{n+m}(\mathbb{D}^2)/(S_n\times S_m))$ and 
$B_m(\mathbb{D}^2\setminus \{n\ \text{points}\})=\pi_1(F_m(\mathbb{D}^2\setminus \{n\ \text{points}\})/S_m)$. As we will see in detail in Section \ref{s2}, one can associate to a periodic orbit $o(x,f)$ together with the set $A$ a braid element $\beta_x\in B_{n,m}$. Using the well-known decomposition $B_{n,m}\cong B_m(\mathbb{D}^2\setminus \{n\ \text{points}\}) \rtimes B_n $, we provide an algebraic characterization of strong Nielsen equivalent classes of periodic orbits, with respect to the set $A$. Let $\phi_{\beta_A}$ denote the endomorphism of $B_m(\mathbb{D}^2\setminus \{n\ \text{points}\})$, induced by $\beta_A\in B_n$, which describes the action of $B_n$ on $B_m(\mathbb{D}^2\setminus \{n\ \text{points}\})$.

\begin{theorem}\label{thm2}
	Let $f:(\mathbb{D}^2, A; \partial \mathbb{D}^2)\to (\mathbb{D}^2, A; \partial \mathbb{D}^2)$ be an orientation-preserving homeomorphism, where $A\subset\text{Int}(\mathbb{D}^2)$ is an $f$-invariant $n$-point set. Let $x,y\in \text{Int}(\mathbb{D}^2)\setminus A$ such that $x,y\in P_m(f)$. The following are equivalent:
	\begin{enumerate}[(i)]
\item $o(x,f) \widesim{\text{\tiny SN}}_A o(y,f)$.
\item $\beta_x=c\cdot \beta_y\cdot c^{-1}$, for $\beta_x, \beta_y\in B_{n,m}\subset B_{n+m}$ and $c\in B_m(\mathbb{D}^2\setminus \{n\ \text{points}\})$.
\item $\beta_{o_x}=\phi_{\beta_A}(c)\cdot \beta_{o_y}\cdot c^{-1},\ \text{for}\ \beta_{o_x}, \beta_{o_y}, c\in B_m(\mathbb{D}^2\setminus \{n\ \text{points}\})$ and $\beta_A\in B_n$.
	\end{enumerate}
	
\end{theorem}
Roughly speaking, the braids $\beta_x, \beta_y$ correspond to the invariant sets $A\cup o(x,f)$ and $A\cup o(y,f)$ respectively, the braids $\beta_{o_x}, \beta_{o_y}$ correspond to the invariant sets $o(x,f)$ and $o(y,f)$ respectively and $\beta_A$ to the invariant set $A$, with respect to the isotopy induced by the suspension flow on the mapping torus $\mathbb{D}^2_f$, as we will see in Section \ref{s2}.\\

\textbf{Application:} We will use this result to show that the forced periodic orbits one obtains from the trace formula by Jiang--Zheng in \cite{boju} belong to different strong Nielsen classes, with respect the given invariant set.\\

The structure of the article is as follows. In Section \ref{s1}, we provide necessary definitions and basic results from Nielsen and braid theory. Moreover, we present how periodic orbits are associated to braid elements. In Section \ref{s2}, we define the strong Nielsen type, commonly called braid type. We focus in the case of $\mathbb{D}^2$, where there are several equivalent ways of defining it, and we conclude by proving Proposition \ref{thm1} and Theorem \ref{thm2}. We explain, in Section \ref{s3}, how from our result one can obtain Theorem 1 of \cite{guaschi}. Thus, Theorem \ref{thm2} can be seen as a generalization of this result to the case of periodic points. Additionally, we make some important remarks comparing periodic Nielsen points with strong Nielsen periodic points, taking into consideration a given invariant $n$-point set. The last section, Section \ref{s4}, is devoted to an application of our result to the Lefschetz trace formula given by Jiang--Zheng in \cite{boju}. We show that the forced braids that they obtain, which could correspond to forced periodic orbits, all belong to different strong Nielsen classes, with respect to a given invariant set. Thus, given a periodic point on $\mathbb{D}^2$, combining their result with ours, one can have an algorithm with output forced periodic orbits, which are not strong Nielsen equivalent, when taking into consideration the given invariant set.

\section{Preliminaries}\label{s1}

In what follows, we assume that $M$ is a compact, connected, orientable surface and that any map $f:M\to M$ is an orientation-preserving homeomorphism.
In this section we give the necessary definitions and theory on periodic points as well as their relation with braids. Moreover, we review the relation between braid group and mapping class group.  

We begin with fixing some notation and giving two equivalent definitions of strong Nielsen equivalence of periodic points. For further details in the theory we direct the reader to \cite{asimov} and \cite{boy}.
A periodic point $x$ is a point for which $f^n(x)=x$ for some $n>0$. The least such $n$ is called the period of the periodic point. Let $P_n(f)$ be the set of all periodic points with period $n$ and let $o(x,f)$ be the orbit of a periodic point $x$. The strong Nielsen equivalence is an equivalence relation on periodic points of a surface homeomorphism. There are two ways to define the strong Nielsen equivalence; one that uses paths in the surface and another that uses the suspension flow. 
     
We recall a couple of notions that we will need for the following definition. Let us denote by $\text{Homeo}(M)$ the group of orientation-preserving homeomorphisms of $M$. An isotopy $f_t:f_0\simeq f_1$ is said to be a deformation of another isotopy $h_t:f_0\simeq f_1$ if the corresponding paths in $\text{Homeo}(M)$ are homotopic with fixed endpoints. A self-isotopy $f_t:f\simeq f$ is called contractible if it is a deformation of the trivial isotopy, which means that the corresponding closed loop in $\text{Homeo}(M)$ is null-homotopic.
     
 \begin{definition}\label{defsne}
Let $x,y\in P_n(f)$. We say that $x,y$ are strong Nielsen equivalent, denoted $x\widesim{\text{\tiny SN}}y$, if there exists a contractible isotopy $f_t:f\simeq f$ and a path $\gamma:[0,1]\to M$ such that $\gamma(0)=x$, $\gamma(1)=y$ and $\gamma(t)\in P_n(f_t)$, for all $t\in [0,1]$.  
 \end{definition}

Note that the strong Nielsen equivalence on periodic points is a stronger equivalence relation compared to the periodic Nielsen equivalence. Recall that given $x,y\in P_n(f)$, $x$ is periodic Nielsen equivalent to $y$ if there exists a path $\gamma:[0,1]\to M$ with $\gamma(0)=x$, $\gamma(1)=y$ and $f^n(\gamma)\sim \gamma$ with fixed endpoints.

This definition gives an equivalence relation on periodic points, while the following one, which uses the suspension flow, defines an equivalence relation on periodic orbits.

Let us recall the definition of a mapping torus of a homeomorphism and of the suspension flow on it.

\begin{definition}\label{flow}
The mapping torus $M_f$ of a homeomorphism $f:M\to M $ is the quotient space ${M\times[0,1]}/{\sim}$, with $(x,1)\sim (f(x),0)$.
However, it is often more convenient to consider $M_f$, as the the quotient space ${M\times\mathbb{R}}/{\sim}$, with $(x,s+1)\sim (f(x),s)$, for every $x\in M$ and $s\in \mathbb{R}$.
To obtain the suspension flow $\psi_t$ on $M_f$, one takes the unit speed flow in the $\mathbb{R}$ direction on $M\times\mathbb{R}$ and then projects it to $M_f$. 
\end{definition} 

Let $p:M\times\mathbb{R}\to M_f$ be the projection and, for any $x\in M$, let $\gamma_x$ denote the orbit of $p(x,0)$ under $\psi_t$.
Notice that when $x$ is a periodic point, then $\gamma_x$ can be viewed as a simple closed curve in $M_f$. In other words, for a periodic point $x\in M$, of period $n\in\mathbb{N}$, $\gamma_x$ is defined to be a loop in $M_f$, given by $\gamma_x(s)=(x, ns)\in M_f$, for $s\in [0, 1]$.

\begin{definition}\label{deforbit}
Let $x,y\in P_n(f)$. The periodic orbits $o(x,f)$ and $o(y,f)$ are strong Nielsen equivalent, denoted $o(x,f)\widesim{\text{\tiny SN}}o(y,f)$, if $\gamma_x$ is freely isotopic to $\gamma_y$ in $M_f$.
\end{definition}

In terms of suspension flow, two periodic orbits $o(x,f)$ and $o(y,f)$ are periodic Nielsen equivalent if $\gamma_x$ is freely homotopic to $\gamma_y$ in $M_f$.

 In Figure \ref{fig}, we illustrate an example given in \cite{boy}, which presents 
the suspension of a map $f$ isotopic to the identity on
the disc with two holes, showing a pair of orbits, of period two, that are
periodic Nielsen equivalent but not strong Nielsen equivalent. It holds that the orbit $o(x_1, f)$ on the left is periodic Nielsen equivalent to the orbit $o(x_1,f)$ on the right, because in homotopy it is allowed to pull strings of a closed loop through
itself. Clearly, this is not allowed in an isotopy.

\begin{figure}[h]
	\centering
	\includegraphics[width=0.8\textwidth]{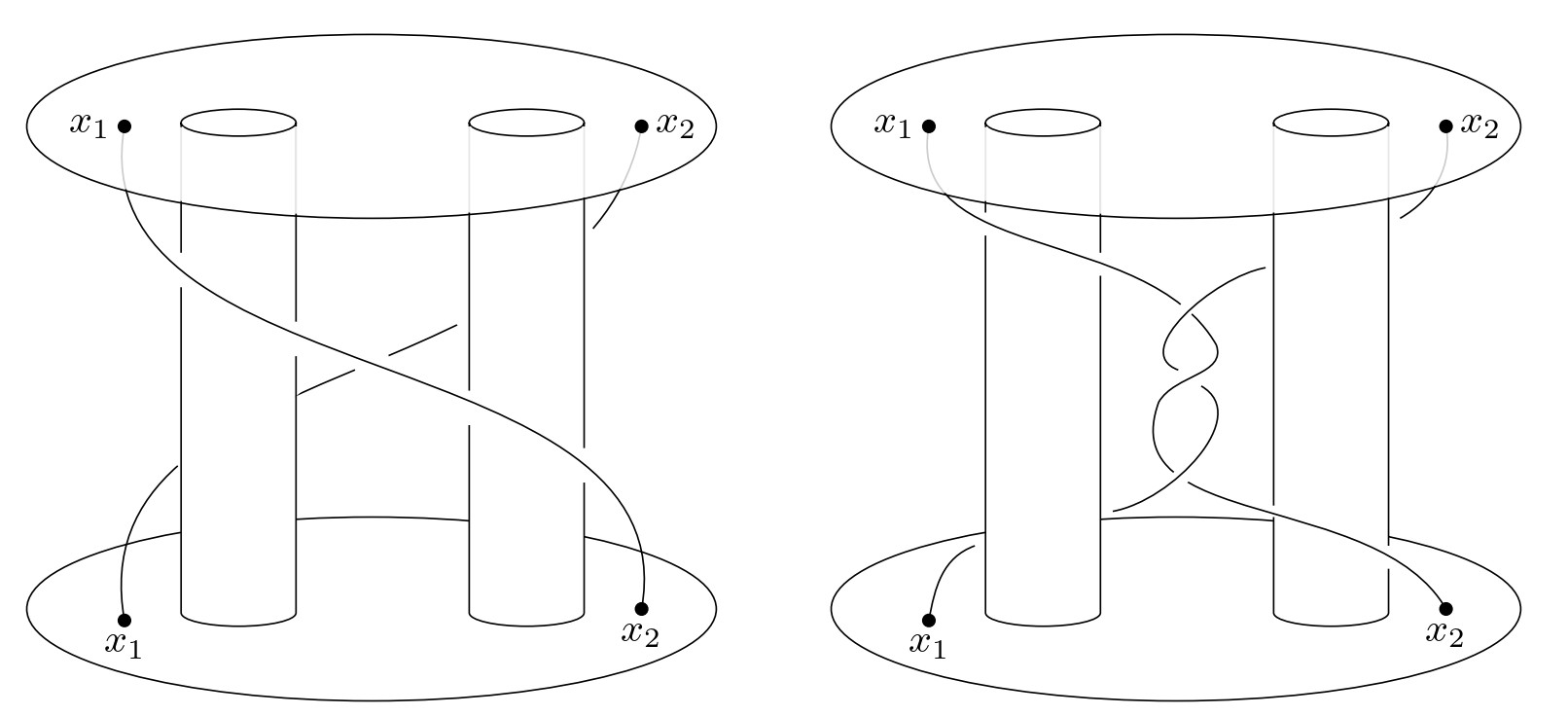}
	\caption{Two orbits of period two that are periodic Nielsen equivalent but not strong Nielsen equivalent.}
	\label{fig}
\end{figure}
 
 The two presented definitions of strong Nielsen equivalence are connected due to the following result, given in \cite{boy}, where part of the proof is in \cite{asimov}.
  
 \begin{proposition}\label{orbit}
Two periodic orbits are strong Nielsen equivalent if and only if there are points from each of the orbits that are strong Nielsen equivalent.
 \end{proposition}
  
 \begin{remark}
 	
 	Note that the classical definition of the Nielsen class is defined only when $n=1$, that is for fixed points, and it only requires that $f(\gamma(t))\sim \gamma(t)$ relative to the endpoints.
 \end{remark}

\begin{remark}\label{same}
The notion of strong Nielsen equivalence, given in Definition \ref{defsne} and the notion of periodic Nielsen equivalence, given below Definition \ref{defsne}, coincide for fixed points, see (\cite{jiang}, Theorem 2.13). That is, for fixed points $x, y$, then $x$ is Nielsen equivalent to $y$ if and only if $x$ is strong Nielsen equivalent to $y$.
\end{remark}
\begin{remark}\label{equiv}
The requirement of a contractible self-isotopy at the definition of strong Nielsen equivalence using paths is necessary to ensure that the definition is equivalent to the one using the suspension flow. Additionally, one would like strong Nielsen equivalence to imply periodic Nielsen equivalence for periodic points, which holds when we have contractible self-isotopy. For further details we refer the reader to (\cite{boy}, 2.4. Remarks).
\end{remark}

\begin{remark}\label{extra}
The reason we assume that $M$ has negative Euler characteristic or that when $M=\mathbb{D}^2$ we restrict our attention to the class of homeomorphisms that preserve pointwise the boundary, is because in both of these cases all self-isotopies of $M$ are contractible; and thus we directly have that the two definitions of strong Nielsen equivalence are equivalent. However, one can define the strong Nielsen equivalence in a more general context.
\end{remark}

We review now the relation between periodic orbits and braids, as well as the relation between mapping classes and braids. An interesting aspect of braid groups is that they can be defined from several viewpoints, such as equivalence classes of geometric braids, as the fundamental group of configuration spaces, as we saw in the introduction, and as trajectories of non-colliding particles. Moreover, they are closely related to mapping class groups. For a more detailed description of these different approaches, we refer the reader to \cite{guaschi2013survey}.

Our main focus is on orientation-preserving homeomorphisms $f:M\to M$ that are isotopic to the identity map $\text{id}_{M}$. Let us fix an isotopy $f_t:f\simeq\text{id}_M$ and let $A$ be an invariant $n$-point set of $f$ in $\text{Int}(M)$. The subset $\beta(A, f_t):=\{(f_t(A), t),\ \text{for}\ t\in[0, 1]\}$ of $M\times [0,1]$ is a geometric braid of $n$-strands, which is just a collection of $n$ strands with basepoints in $A$. For simplicity, we denote it by $\beta_A$, and we refer to it as the $n$-braid associated to $A$. The strands of $\beta_A$ appear naturally by following $A$ under the isotopy. More precisely, each strand of $\beta_A$ connects two points $(a_i,0)$ and $(f(a_i), 1)$, for $1\leq i\leq n$ and $a_i\in A$.
 Note that there is a natural map that associates a permutation $s(\beta_A)\in S_n$, where $S_n$ is the symmetric group of degree $n$, to an $n$-braid $\beta_A$, that corresponds to the permutation that occurs to the points in $A$ under the $n$-braid $\beta_A$. In the particular case where $A$ is a periodic orbit of a periodic point of period $n$, the $n$-braid associated to $A$ corresponds to an $n$-cycle in $S_n$. In the case of $\mathbb{D}^2$, the group formed by the isotopy classes of geometric $n$-braids is equivalent to the Artin braid group $B_n$. We recall the classical presentation of the Artin braid group $B_n$:

$$B_n=\langle \sigma_1, \dots, \sigma_{n-1}\ |\ \sigma_i\sigma_j=\sigma_j\sigma_i,\ \text{for} \ |i-j|>1,\ \sigma_i\sigma_{i+1}\sigma_i = \sigma_{i+1}\sigma_i\sigma_{i+1},\ \text{for all} \ 1\leq i\leq n-2 \rangle.$$

The generator $\sigma_i$ can be seen geometrically as the braid with a single positive crossing of the $i^{th}$ strand with the $(i+1)^{st}$ strand, while all other strands remain vertical. We will investigate the relation between braids and periodic orbits of the 2-disc further in Section \ref{s2}.

 We focus now in the case when $M$ is the 2-disc, since this is the main interest of this work, and we present the relation between mapping classes and braids.
 The collection of isotopy classes of orientation preserving homeomorphisms on $\mathbb{D}^2$ that fix the boundary $\partial \mathbb{D}^2$ pointwise, with the operation of composition is called the mapping class group of $\mathbb{D}^2$ and is denoted by MCG$(\mathbb{D}^2; \partial \mathbb{D}^2)$. Recall that, due to Alexander's trick, any orientation preserving homeomorphism of $\mathbb{D}^2$ that fixes the boundary pointwise is isotopic to the identity map. As a result it follows that MCG$(\mathbb{D}^2; \partial\mathbb{D}^2)$ is trivial. If in addition we consider isotopy classes, relative to an $n$-point set $A\subset \text{Int}(\mathbb{D}^2)$, which means that both all homeomorphisms and isotopies must leave $A$ invariant, then the corresponding mapping class group is denoted by MCG$(\mathbb{D}^2, A; \partial \mathbb{D}^2)$. It is well-known that the group MCG$(\mathbb{D}^2, A; \partial \mathbb{D}^2)$ can be identified with Artin's braid group on $n$-strands $B_n$. That is, $$\text{MCG}(\mathbb{D}^2, A;\partial\mathbb{D}^2)\cong B_n.$$ One can see that they are indeed naturally related in the following way. Let $\beta\in B_n$ a braid where the points in $A$ play the role of the endpoints of $\beta$. One can extend the velocity vector ﬁeld of the $n$-strands to $\mathbb{D}^2\times[0,1]$ and a desired isotopy is generated by this vector ﬁeld. Intuitively, one can
 think of sliding $\mathbb{D}^2$ down the braid to obtain an isotopy. On the other hand, given $f:\mathbb{D}^2\to \mathbb{D}^2$ with $f(A)=A$, one can associate a braid $\beta\in B_n$ with $A$ by fixing an isotopy $f_t:f\simeq\text{id}_{\mathbb{D}^2}$. The strands of the braid $\beta$ appear naturally following $A$ under the isotopy $f_t$. For more details see \cite{birman}.

For completion let us conclude this section reviewing the Nielsen--Thurston classification \cite{classification}, \cite{thurston} and showing its connection with braids.

\begin{theorem}
For every homeomorphism $f: (M,A)\to (M, A)$ of a compact surface $M$, with $M\setminus A$ having negative Euler characteristic, there exists a homeomorphism $\phi$, isotopic to $f$ relative to $A$, the Thurston representative of the isotopy class, that satisfies one of the following:
\begin{enumerate}[(i)]
\item $\phi$ is finite order, i.e. there exists $k\in \mathbb{N}$ such that $\phi^k=\text{id}$.
\item $\phi$ is pseudo-Anosov, i.e it preserves a transverse pair of measured singular foliations, expanding the measure uniformly along the leaves of one foliation and contracting it uniformly, by the same factor, along the leaves of the other.
\item $\phi$ is reducible, i.e. there exists a collection of pairwise disjoint simple closed curves $\gamma=\{\gamma_1,\dots,\gamma_k\}$ in $\text{Int}(M)\setminus A$, called reducing curves for $\phi$, such that $\phi(\gamma)=\gamma$. Moreover, $\gamma$ has a $\phi$-invariant open tubular neighborhood $U$, which does not intersect with the set $A$, such that each component of $M\setminus (U\cup A)$ has negative Euler characteristic and the restriction of an appropriate iterate of $\phi$ to each component of $M\setminus U$ is either finite order or pseudo-Anosov relative to $(M\setminus U)\cap A$. 
\end{enumerate}
\end{theorem} 

Combining the fact that MCG$(\mathbb{D}^2, A; \partial \mathbb{D}^2)\cong B_n$ with the Nielsen--Thurston classification, it follows that braids appear naturally into three types: periodic, pseudo-Anosov and reducible.

\section{Main results}\label{s2}

In this section we present the notion of strong Nielsen type, which was introduced by Boyland \cite{boy} and Matsuoka \cite{matsuoka2}, and we prove Proposition \ref{thm1}. Moreover, we introduce the definition of strong Nielsen equivalence, relative to an $f$-invariant $n$-point set, of periodic orbits and we give a proof of Theorem \ref{thm2}. 
	
	Let $f: M\to M$ be an orientation-preserving homeomorphism, isotopic to the identity, of a compact connected orientable surface $M$ and let $X_n\subset\text{Int}(M)$ be an $n$-point set. Given a periodic orbit $o(x,f)$ of a point $x$ of period $n$, choose a homeomorphism $h_x: M\to M$, isotopic to the identity, such that $h_x(o(x,f))=X_n$.
	
	\begin{definition}\label{defsn}
The strong Nielsen type of the orbit $o(x,f)$, denoted by $snt(x,f)$, is the conjugacy class of $[h_xfh_x^{-1}]$ in MCG$(M, X_n)$, where $[h_xfh_x^{-1}]$ stands for the isotopy class of $h_xfh_x^{-1}$.
	\end{definition}
	
	 Note that we consider the conjugacy class of $[h_xfh_x^{-1}]$ in MCG$(M, X_n)$ in order to have independence of the choice of the homeomorphism $h_x$. Actually, the strong Nielsen type of an orbit is the isotopy class of $f$ relative to the orbit, however we use the homeomorphism $h_x$ in order to send the orbits to a common model. For orientation-preserving homeomorphisms of the disk that preserve the boundary pointwise, due to the identification MCG$(\mathbb{D}^2, X_n; \partial \mathbb{D}^2)\cong B_n$, a strong Nielsen type of an orbit is commonly called braid type and in this case we use the notation $snt(x,f)=bt(x,f)$.
	
\begin{proof}[Proof of Proposition \ref{thm1}]
		The first implication follows from Lemma 8 of \cite{hall} assuming $f=g$.

		To prove the inverse implication let $x,y\in P_n(f)$ and let us assume that $snt(x,f)=snt(y,f)$. Thus, up to conjugation in MCG$(M,X_n)$, $[h_xfh_x^{-1}]=[h_yfh_y^{-1}]$. That is, there exists an isotopy between $h_xfh_x^{-1}$ and $h_yfh_y^{-1}$ in $MCG(M, X_n)$. Thus, there exists a continuous map $H:[0,1]\times (M, X_n)\to (M, X_n)$, such that $H_0=h_xfh_x^{-1}$, $H_1=h_yfh_y^{-1}$ and $H_t:(M, X_n)\to (M,X_n)$ is a homeomorphism, for every $t\in [0,1]$. From this isotopy we can further obtain a self-isotopy of $f$. To do so, we first choose a family of $n$-points sets $K_t\subset \text{Int}(M)$, for $t\in [0,1]$, such that $K_0=o(x,f)$ and $K_1=o(y,f)$. In addition, choose $h_t:M\to M$ homeomorphisms isotopic to the identity, such that $h_0=h_x$ and $h_1=h_y$ and $h_t(K_t)=X_n$, for every $t\in [0,1]$. We define now the family of maps $F:[0,1]\times M\to M$ as $F_t=h_t^{-1}H_th_t$, for $t\in [0,1]$.
		Notice that $F_0=h_0^{-1}H_0h_0=f$, $F_1=h_1^{-1}H_1h_1=f$ and that $F_t=h_t^{-1}H_th_t$ is a homeomorphism of $M$, for every $t\in [0,1]$. Moreover, the choice of the family of the $n$-point sets $K_t$ as well as the choice of the homeomorphisms $h_t$ is done in a continuous way, in order to achieve continuity of $F_t=h_t^{-1}H_th_t$ with respect to the parameter $t$. Thus, by definition, $F_t$ is a continuous family of homeomorphisms of $M$ and thus a self-isotopy of $f$.
		
		Having defined the self-isotopy $F_t$ it remains to construct a path $\gamma:[0,1]\to M$ such that $\gamma(0)=x$, $\gamma(1)=y$ and $\gamma(t)\in P_n(F_t)$, for all $t\in [0,1]$. Note that, without loss of generality, we can assume that $h_x(x)$ and $h_y(y)$ have the same image in $X_n$, since we consider the homeomorphisms $h_xfh_x^{-1}$ and $h_yfh_y^{-1}$ up to conjugation in MCG$(M, X_n)$, in order to have independence of the choice of $h_x$ and $h_y$. That is, $h_x(x)=\alpha\in X_n$ and $h_y(y)=\alpha\in X_n$. It follows that $H_0^n(\alpha)=h_xf^nh_x^{-1}(\alpha)=h_xf^n(x)=h_x(x)=\alpha$ and similarly $H_1^n(\alpha)=h_yf^nh_y^{-1}(\alpha)=\alpha$. In particular, $H^n_t(\alpha)=\alpha\in X_n$, for all $t\in [0,1]$, since all homeomorphisms $H_t$ belong in the same isotopy class in MCG$(M, X_n)$. 
		Let us define now a path $\gamma:[0,1]\to M$ connecting $x$ and $y$ as follows: $\gamma(t)=h^{-1}_t(\alpha)\in M$. It holds that $\gamma(0)=h^{-1}_0(\alpha)=h_x^{-1}(\alpha)=x$ and $\gamma(1)=h^{-1}_1(\alpha)=h_y^{-1}(\alpha)=y$. Moreover, $F^n_t(\gamma(t))=F_t^n(h^{-1}_t(\alpha))=
		h_t^{-1}H^n_th_t(h^{-1}_t(\alpha))=h_t^{-1}H^n_t(\alpha)=h^{-1}_t(\alpha)=\gamma(t)$, for all $t\in [0, 1]$. Therefore, $\gamma(t)\in P_n(F_t)$, for all $t\in [0, 1]$, which concludes the proof.
		
	\end{proof}
		
	Let us focus now in the case of the 2-disc and in orientation-preserving homeomorphisms $f:\mathbb{D}^2\to\mathbb{D}^2$ that fix the boundary pointwise and have a periodic point $x\in P_m(\mathbb{D}^2)$ in the interior of $\mathbb{D}^2$. We saw in Definition \ref{defsn} that the braid type of $o(x,f)$ is defined as the conjugacy class of $[h_xfh_x^{-1}]$ in MCG$(\mathbb{D}^2, X_m; \partial \mathbb{D}^2)\cong B_m$. We will see now three more equivalent ways of defining the braid type of a periodic orbit. First, let us fix an isotopy $f_t:f\simeq id_{\mathbb{D}^2}$ and let $\beta_{o_x}$ be the geometric $m$-braid associated to $o(x,f)$. We recall that $\beta_{o(x,f)}$ is defined as the set of $m$ strands $\{(f_t(o(x,f)), t),\ \text{for}\ t\in[0, 1]\}$ in $\mathbb{D}^2\times[0,1]$. We define the free isotopy class of $\beta_{o_x}$ to be the braid type of $o(x,f)$. We recall that two geometric $m$-braids are said to be freely isotopic if one can be continuously deformed to the other through geometric $m$-braids, without fixing the base points during the deformation. Furthermore, two geometric $m$-braids are freely isotopic if and only if their corresponding loops in the $n^{th}$ unordered configuration space $F_m(\mathbb{D}^2)/S_m$ are freely homotopic. That is, there is a homotopy between the two loops, which let the base point vary. Thus, a braid type can also be deﬁned as a free homotopy class of loops in $F_m(\mathbb{D}^2)/S_m$. Finally, due to the bijective correspondence of the set of free homotopy classes of loops in $F_m(\mathbb{D}^2)/S_m$ and the conjugacy classes of the fundamental group $\pi_1(F_m(\mathbb{D}^2)/S_m)$ for any base point, a braid type can also be seen as a conjugacy class in the group $\pi_1(F_n(\mathbb{D}^2)/S_m)=B_m$. Therefore, the braid type of $o(x,f)$ can equivalently be defined as the conjugacy class of $\beta_{o_x}$ in $B_m$. 
	We direct the reader to \cite{matsuoka} for further details on the equivalence of the definitions of the braid type.
	
	As a consequence of Proposition \ref{thm1} we obtain that the problem of distinguishing strong Nielsen equivalence classes of an orientation-preserving homeomorphism of the 2-disc, turns out to be a conjugacy problem in the braid group. More precisely, we obtain Corollary \ref{corollary}, which asserts the following:
	Let $f:\mathbb{D}^2\to\mathbb{D}^2$ be an orientation-preserving homeomorphism that fixes the boundary pointwise and let $x,y\in\text{Int}(\mathbb{D}^2)$ such that $x,y\in P_m(\mathbb{D}^2)$. Then $x\widesim{\text{\tiny SN}}y$ if and only if, for some $c\in B_m$, $\beta_{o_x}=c\cdot\beta_{o_y}\cdot c^{-1}$, where $\beta_{o_x}, \beta_{o_y}\in B_m$. We recall that $\beta_{o_x}$ and $\beta_{o_y}$ are the $m$-braids associated to $o(x,f)$ and $o(y,f)$ respectively, for a fixed isotopy $f_t:f\simeq id_{\mathbb{D}^2}$.

From this result we observe that the strong Nielsen equivalence of periodic orbits does not take into consideration any other possible $f$-invariant sets in the interior of $\mathbb{D}^2$. For this reason, in what follows we will examine the notion of strong Nielsen equivalence, with respect to an $f$-invariant set in $\text{Int}(\mathbb{D}^2)$.
 	
Let $f:(\mathbb{D}^2, A; \partial\mathbb{D}^2)\to (\mathbb{D}^2, A; \partial\mathbb{D}^2)$ be an orientation-preserving homeomorphism and $A\subset \text{Int}(\mathbb{D}^2)$ be an $f$-invariant $n$-point set, that is, $f(A)=A$. Suppose that $x, y\in \text{Int}(\mathbb{D}^2)\setminus A$ such that $x, y\in P_m(f)$. Moreover, let us fix an isotopy $f_t:f\simeq\text{id}_{\mathbb{D}^2}$. In Section \ref{s1}, we saw that we can associate an $n$-braid, $\beta_A\in B_n$, to the $f$-invariant $n$-point set $A$. In particular, $\beta_A=\{(f_t(A), t),\ \text{for}\ t\in[0, 1]\}$ is a geometric braid of $n$-strands in $\mathbb{D}^2\times[0,1]$. Let $B_{n,m}$ be the subgroup of $B_{n+m}$ that consists of the braids whose strands are divided into two blocks, one of $n$-strands and the other of $m$-strands, possibly linked. Now, every periodic point of $f$ of period $m$, in the interior of the 2-disc that does not belong in $A$, defines an element of $B_{n,m}$, considered as an extension of the element $\beta_A$ by adding $m$ strands, which encode the topological interaction of the periodic orbit with $A$.
 More precisely, the periodic point $x$ defines the braid $\beta_x:=\{(f_t(A\cup o(x,f)), t),\ \text{for}\ t\in[0, 1]\}$. By definition, $\beta_x$ is an $(n+m)$-braid in $\mathbb{D}^2\times[0,1]$, which belongs to $B_{n,m}$, since, under $f_t$,
 the points in $A$ are connected with points in $A$ and the points in $o(x,f)$ are connected with points in $o(x,f)$. Similarly, the periodic point $y$ defines the braid $\beta_y:=\{(f_t(A\cup o(y,f)), t),\ \text{for}\ t\in[0, 1]\}$, an element in $B_{n,m}$.
 
 Let us consider now the homomorphism $p:B_{n,m} \to B_n$, considered geometrically as removing or forgetting the last $m$-stands. For example, $p(\beta_x)=\beta_A$ and $p(\beta_y)=\beta_A$. The kernel of this map is the group $B_m(\mathbb{D}^2\setminus \{n\ \text{points}\})$. It can be seen as the group of the isotopy classes of $m$-braids in $(\mathbb{D}^2\setminus \{n\ \text{points}\})\times [0,1]$ or equivalently as the fundamental group of the quotient of the unordered configuration space $F_m(\mathbb{D}^2\setminus \{n\ \text{points}\})/S_m$. That is,
 $B_m(\mathbb{D}^2\setminus \{n\ \text{points}\})=\pi_1(F_m(\mathbb{D}^2\setminus \{n\ \text{points}\})/S_m)$, as we mentioned in the introduction. These groups fit into the following short exact sequence:
$$1\rightarrow B_m(\mathbb{D}^2\setminus \{n\ \text{points}\})\rightarrow B_{n,m}\xrightarrow{p} B_n\rightarrow 1,$$
	where one can consider $p$ as forgetting the last $m$ strands. It is well-known that this short exact sequence splits and thus we have $B_{n,m}\cong B_m(\mathbb{D}^2\setminus\{n \ \text{points}\}) \rtimes B_n $. One can think of the section $\iota: B_n\to B_{n,m}$ of $p$ as adding $m$ vertical strands at the end of a braid $\beta\in B_n$. Moreover, under the inclusion $i: B_m(\mathbb{D}^2\setminus \{n\ \text{points}\})\to B_{n,m}$, the elements of $B_m(\mathbb{D}^2\setminus \{n\ \text{points}\})$ can be seen as elements in $B_{n,m}$, by adding $n$-vertical strands whose endpoints correspond to the $n$ punctures. 
	As a result, an element $b\in B_{n,m}$ can be uniquely written as $b=\iota(\beta)\cdot i(\gamma)$, for $\beta\in B_m$ and $\gamma\in B_m(\mathbb{D}^2\setminus\{n \ \text{points}\})$. For convenience, we keep the same notation and we denote $i(\gamma)$ as $\gamma$ and $\iota(\beta)$ as $\beta$. Thus, $b=\beta\cdot\gamma\in B_{n,m}$. Furthermore, for any $\beta\in B_n$ we denote by $\phi_{\beta}:B_m(\mathbb{D}^2\setminus \{n\ \text{points}\}) \to B_m(\mathbb{D}^2\setminus \{n\ \text{points}\})$ the action of $B_n$ on  $B_m(\mathbb{D}^2\setminus \{n\ \text{points}\})$, based on the decomposition $B_{n,m}\cong B_m(\mathbb{D}^2\setminus\{n \ \text{points}\}) \rtimes B_n $, which is defined as follows:
	\begin{equation}\label{end}
		\phi_{\beta}(\gamma)=\beta^{-1}\cdot\gamma\cdot\beta,\ \text{for any}\ \gamma \in B_m(\mathbb{D}^2\setminus \{n\ \text{points}\}).
	\end{equation}
Note that $\phi_{\beta}(\gamma)\in B_m(\mathbb{D}^2\setminus \{n\ \text{points}\})$, since $B_m(\mathbb{D}^2\setminus \{n\ \text{points}\})$ is a normal subgroup of $B_{n,m}$. For more insights about the splitting of the short exact sequence we refer the reader to \cite{guaschi2013survey}.
	
	In our setting, we saw that $\beta_x, \beta_y\in B_{n,m}$ can be considered as extensions of $\beta_A\in B_n$ by the orbit of $x$ and $y$ respectively. That is,
	$\beta_x\in B_{n,m}$ and $\beta_y\in B_{n,m}$ can be written uniquely as
	$\beta_x=\beta_A\cdot \beta_{o_x}$ and $\beta_y=\beta_A\cdot \beta_{o_y}$, respectively, where $\beta_A \in B_n$ is the $n$-braid $\{(f_t(A), t),\ \text{for}\ t\in[0, 1]\}$ in $\mathbb{D}^2\times[0,1]$, $\beta_{o_x}\in B_m(\mathbb{D}^2\setminus\{n \ \text{points}\})$ is the $m$-braid $\{(f_t(o(x,f)), t),\ \text{for}\ t\in[0, 1]\}$ in $(\mathbb{D}^2\setminus A)\times [0,1]$ and similarly $\beta_{o_y}\in B_m(\mathbb{D}^2\setminus\{n \ \text{points}\})$ is the $m$-braid $\{(f_t(o(y,f)), t),\ \text{for}\ t\in[0, 1]\}$ in $(\mathbb{D}^2\setminus A)\times [0,1]$. Therefore, the braid $\beta_A$ induces an endomorphism $\phi_{\beta_A}$ of	
     $B_m(\mathbb{D}^2\setminus\{n \ \text{points}\})$. That is, $\phi_{\beta_A}(\beta_{o_x})=\beta_A^{-1}\cdot \beta_{o_x}\cdot\beta_A\in B_m(\mathbb{D}^2\setminus\{n \ \text{points}\})$ and $\phi_{\beta_A}(\beta_{o_y})=\beta_A^{-1}\cdot \beta_{o_y}\cdot\beta_A\in B_m(\mathbb{D}^2\setminus\{n \ \text{points}\})$.
     
     We are ready now to define the strong Nielsen equivalence of periodic orbits, with respect to an $f$-invariant $n$-point set, and then to prove Theorem \ref{thm2}.
     
     \begin{definition}
     Let $f:(M, A)\to (M, A)$ be an orientation-preserving homeomorphism, where $A\subset\text{Int}(M)$ an $f$-invariant $n$-point set. Let $x,y\in \text{Int}(M)\setminus A$ such that $x,y\in P_m(f)$. The periodic orbits of $x$ and $y$ are strong Nielsen equivalent, with respect to $A$, denoted by $o(x,f)\widesim{\text{\tiny SN}}_A o(y,f)$, if $\gamma_x$ is freely isotopic to $\gamma_y$ in $M_f$, while taking into consideration and keeping fixed the orbits $\gamma_{a_j}$ of $p(a_j,0)$ in $M_f$, for all $a_j\in A$. 
     \end{definition}
   
   Observe that it is possible to have $\gamma_{a_j}=\gamma_{a_k}$, for $j\neq k$. In other words, the closed orbits determined by $a, f(a), \dots, f^{k-1}(a)$, where $k\in\mathbb{N}$ is the period of a point $a\in A$, are the same closed curves in $M_f$ with different parametrization. In particular, the number of distinct simple closed curves that occur from $A$ corresponds to the number of cycles and fixed points that arise from the permutation of the set $A$ under $f$. Note that for $A=\emptyset$ then we obtain the standard definition of strong Nielsen equivalence of periodic orbits, Definition \ref{deforbit}. 

It is important to note that, in general, a flow $\phi_t(x): M\times \mathbb{R} \to M$ itself defines naturally an isotopy between the homeomorphism $\phi_1(x):M\to M$
and $\text{id}_M$. In the case of our interest, where $M=\mathbb{D}^2$ and thus $f$ is isotopic to the identity we consider the suspension flow $\psi_t$ on the mapping torus $\mathbb{D}^2_f$. It follows that $\psi_t$ defines naturally an isotopy between the homeomorphism $\psi_1=f$ and $\text{id}_{\mathbb{D}^2}$. During the following proof, the geometric braids of $f$-invariant sets that are considered, are considered with respect to the isotopy induced by the suspension flow on $\mathbb{D}^2_f$. We recall that we have given the definition of the mapping torus and of the suspension flow in Definition \ref{flow}.  
   
	\begin{proof}[Proof of Theorem \ref{thm2}]
		
			Let $f:(\mathbb{D}^2, A; \partial \mathbb{D}^2)\to (\mathbb{D}^2, A; \partial \mathbb{D}^2)$ be an orientation-preserving homeomorphism, where $A\subset\text{Int}(\mathbb{D}^2)$ is an $f$-invariant $n$-point set and $x,y\in \text{Int}(\mathbb{D}^2)\setminus A$ such that $x,y\in P_m(f)$.
			
		 Suppose that $o(x,f)\widesim{\text{\tiny SN}}_A o(y,f)$. Then, the closed curves $\gamma_x$ and $\gamma_y$ are freely isotopic, while keeping fixed the closed curves $\gamma_{a_j}$, for all $a_j\in A$. Since the homeomorphism $f$ is isotopic to the identity we know that $M_f$ is a solid torus, which we consider that it is trivially embedded in $\mathbb{S}^3$, and in particular $M_f$ is homeomorphic to $\mathbb{D}^2\times \mathbb{S}^1$. The closed curves $\gamma_x, \gamma_y$ and $\gamma_{a_j}$, for all $a_j\in A$, can be considered as one-component links, obtained from the closure of the braids $\beta_{o_x}, \beta_{o_y}$ and $\beta_A$, respectively, which correspond to the $m$-braid associated to $o(x,f), o(y,f)$ and to the $n$-braid associated to $A$, with respect to the isotopy $\psi_t:f\simeq id_{\mathbb{D}^2}$, induced by the suspension flow on $\mathbb{D}^2_f$. That is, $\beta_{o_x}=\{\psi_t(o(x,f), t),\ \text{for}\ t\in [0,1]\}$, $\beta_{o_y}=\{\psi_t(o(y,f), t),\ \text{for}\ t\in [0,1]\}$ and $\beta_{A}=\{\psi_t(A, t),\ \text{for}\ t\in [0,1]\}$. Another way of seeing that, consider the mapping torus $\mathbb{D}^2_f$ together with the time-1 orbit curves of all points in $o(x,f), o(y,f)$ and A. If we cut $\mathbb{D}^2_f$ along its gluing, that is, along $\mathbb{D}^2\times\{0\}$, we will obtain $\mathbb{D}^2\times[0, 1]$ and in the interior we will have the braids $\beta_{o_x}, \beta_{o_y}$ and $\beta_A$, whose closure coincide with $\gamma_x, \gamma_y$ and $\gamma_{a_j}$, for all $a_j\in A$.  
		 It is possible that $\beta_{o_x}$ and $\beta_{o_y}$ are linked with $\beta_{A}$ and as we are interested in the way they are linked we will not consider the braids $\beta_{o_x}$ and $\beta_{o_y}$ alone, but along with the braid $\beta_{A}$. Thus, let us consider the two extensions of $\beta_A$, that is, the $(n+m)$-braids $\beta_x=\beta_A\cdot \beta_{o_x}$ and $\beta_y=\beta_A\cdot \beta_{o_y}$. More precisely, $\beta_x=\{\psi_t(A\cup o(x,f), t),\ \text{for}\ t\in [0,1]\}$ and $\beta_y=\{\psi_t(A\cup o(y,f), t),\ \text{for}\ t\in [0,1]\}$. We saw that equivalently we can consider the braids $\beta_x, \beta_y$ as elements of the group $B_{n,m}\subset B_{n+m}$, $\beta_A$ as an element in the group $B_n$ and $\beta_{o_x}, \beta_{o_y}$ as elements in $B_m(\mathbb{D}^2\setminus\{n\ \text{points}\})$. It is well-known that the closure of two braids $\beta, \bar{\beta}$ in $B_n$ give isotopic links in $\mathbb{D}^2\times\mathbb{S}^1$ if and only if they are conjugate in $B_n$, see \cite{kassel}. The conjugacy in $B_n$ corresponds to the change of base point of $\beta$ to obtain $\bar{\beta}$.
		 From hypothesis it holds that $\gamma_x$ and $\gamma_y$ are freely isotopic, while keeping fixed the closed curves $\gamma_{a_j}$, for all $a_j\in A$. In other words, $\bigcup_{a_j\in A}\gamma_{a_j}\cup\gamma_x$ is isotopic to $\bigcup_{a_j\in A}\gamma_{a_j}\cup\gamma_y$, allowing only the base points of $\gamma_x$ and $\gamma_y$ vary during the deformation. From the closure of $\beta_x=\beta_A\cdot \beta_{o_x}\in B_{n,m}$ and $\beta_y=\beta_A\cdot \beta_{o_y}\in B_{n,m}$, we obtain $\bigcup_{a_j\in A}\gamma_{a_j}\cup\gamma_x$ and $\bigcup_{a_j\in A}\gamma_{a_j}\cup\gamma_y$ in $\mathbb{D}^2\times \mathbb{S}^1$ respectively.  
         Therefore, it follows that $\beta_x=\beta_A\cdot \beta_{o_x}$ and $\beta_y=\beta_A\cdot \beta_{o_y}$ are conjugate in $B_{n,m}$ by an element $c$ in $B_m(\mathbb{D}^2\setminus\{n\ \text{points}\})$, which corresponds to the change of base point of $\beta_{o_y}$ in order to obtain $\beta_{o_x}$, while keeping $\beta_{A}$ fixed. That is, $\beta_x=c\cdot \beta_y\cdot c^{-1}$, for a $c\in  B_m(\mathbb{D}^2\setminus\{n\ \text{points}\})$.
		 Let us assume now that $\beta_x=c\cdot \beta_y\cdot c^{-1}$, for some $c\in  B_m(\mathbb{D}^2\setminus\{n\ \text{points}\})$, where $\beta_x=\{\psi_t(A\cup o(x,f), t),\ \text{for}\ t\in [0,1]\}$ and $\beta_y=\{\psi_t(A\cup o(y,f), t),\ \text{for}\ t\in [0,1]\}$. The conjugacy of $\beta_y=\beta_A\cdot \beta_{o_y}$ by an element in $B_m(\mathbb{D}^2\setminus\{n\ \text{points}\})$ corresponds to the change of base point of $\beta_{o_y}$ to obtain $\beta_{o_x}$, while keeping fixed $\beta_A$. The closure of
		 $\beta_x=\beta_A\cdot \beta_{o_x}\in B_{n,m}$ and $\beta_y=\beta_A\cdot \beta_{o_y}\in B_{n,m}$ gives us $\bigcup_{a_j\in A}\gamma_{a_j}\cup\gamma_x$ and $\bigcup_{a_j\in A}\gamma_{a_j}\cup\gamma_y$ in $\mathbb{D}^2\times \mathbb{S}^1$ respectively. 
		 Therefore, $\bigcup_{a_j\in A}\gamma_{a_j}\cup\gamma_x$ and $\bigcup_{a_j\in A}\gamma_{a_j}\cup\gamma_y$ are isotopic in $\mathbb{D}^2\times \mathbb{S}^1$, while letting the base point of $\gamma_x$ and $\gamma_y$ vary, but keeping $\bigcup_{a_j\in A}\gamma_{a_j}$ fixed. We deduce that $\gamma_x$ and $\gamma_y$ are freely isotopic, while keeping $\gamma_{a_j}$, for all $a_j\in A$, fixed. We conclude that $o(x,f)\widesim{\text{\tiny SN}}_A o(y,f)$.
			
		We continue with the proof of the second equivalence of the theorem. 
		Suppose that $\beta_x=c\cdot\beta_y\cdot c^{-1}$, for $\beta_x, \beta_y\in B_{n,m}$ and $c\in B_m(\mathbb{D}^2\setminus \{n\ \text{points}\})$. By definition of $\beta_x$ and $\beta_y$, we have $\beta_A\cdot \beta_{o_x}=c\cdot (\beta_A\cdot \beta_{o_y})\cdot c^{-1}=\beta_A\cdot \phi_{\beta_A}(c)\cdot \beta_{o_y}\cdot c^{-1}$. The second equality follows from the definition of the endomorphism $\phi_{\beta_A}$, described in \refeq{end}. Thus, $\beta_{o_x}=\phi_{\beta_A}(c)\cdot \beta_{o_y}\cdot c^{-1}$, for some $c\in B_m(\mathbb{D}^2\setminus \{n\ \text{points}\})$. On the other hand, suppose that $\beta_{o_x}=\phi_{\beta_A}(c)\cdot \beta_{o_y}\cdot c^{-1}$, for some $c\in B_m(\mathbb{D}^2\setminus \{n\ \text{points}\})$. It follows that $\beta_A\cdot \beta_{o_x}=\beta_A\cdot \phi_{\beta_A}(c)\cdot \beta_{o_y}\cdot c^{-1}$ and from the definition of the endomorphism $\phi_{\beta_A}$ we obtain that $\beta_A\cdot \beta_{o_x}=c\cdot\beta_A\cdot \beta_{o_y}\cdot c^{-1}$, which implies that $\beta_x=c\cdot\beta_y\cdot c^{-1}$.		
		This proves the second equivalence and concludes the proof of the theorem.
		
	\end{proof}
	We conclude this section with few remarks.
	
	\begin{remark}
If $\beta_A$ is the identity element in $B_n$, that is, a collection of $n$ vertical strands, then it is like considering $\mathbb{D}^2$ with $n$-holes in the interior. Moreover, in this case it follows that $\beta_{o_x}=c\cdot\beta_{o_y}\cdot c^{-1}$, where the braids $\beta_{o_x}, \beta_{o_y}$ and $c$ can been considered as geometric braids in $\mathbb{D}^2$ with $n$-holes in the interior.
	\end{remark}
	
	\begin{remark}
If $A=\emptyset$ then Theorem \ref{thm2} coincides with Corollary \ref{corollary}.
	\end{remark}

\begin{remark}
The braid type of an $f$-periodic $n$-orbit in $\text{Int}(\mathbb{D}^2)$ is unique up to multiples of the full twist braid $\Delta_n^2\in B_n$, see \cite{matsuoka3}. The braid $\Delta_n^2=(\sigma_1\dots\sigma_{n-1})^n$ generates the center of $B_n$, for $n\geq 3$, and geometrically corresponds to the braid of $n$ vertical strands after a complete rigid rotation of the 2-disc. Thus, choosing a different isotopy results the same braid type, up to multiples of the full twist $\Delta_n^2$. In particular, if two periodic orbits have the same braid type, under one isotopy, changing the isotopy will not affect this result. That is, these two periodic orbits will still have the same braid type.
\end{remark}	
	
\section{Remarks}\label{s3}

In this section we will make some important remarks regarding Theorem $\ref{thm2}$.
We recall the following result.

\begin{theorem}[Guaschi, \cite{guaschi}]\label{gu}
Let $\phi$ be an automorphism of $\mathbb{F}_n$ induced by a braid $\beta\in B_n$. Then $u,v\in \mathbb{F}_n$
are $\phi$-conjugate if and only if $\beta_u$ and $\beta_v$ are conjugate in $B_{n,1}$ via an element of $\mathbb{F}_n$.
\end{theorem}

\begin{remark}
Applying Theorem \ref{thm2} for $m=1$, we deduce that the fixed points $x, y$ are strong Nielsen equivalent, with respect to $A$, if and only if
$\beta_{o_x}=\phi_{\beta_A}(c)\cdot \beta_{o_y}\cdot c^{-1},\ \text{for some}\ c\in B_1(\mathbb{D}^2\setminus \{n\ \text{points}\})$. Notice that 
$B_1(\mathbb{D}^2\setminus \{n\ \text{points}\})\cong \mathbb{F}_n$, where $\mathbb{F}_n$ is the free group of rank $n$. Moreover, $\phi_{\beta_A}$ turns out to be an endomorphism of $\mathbb{F}_n$. Thus, another way of interpreting Theorem \ref{gu}, could be the following: Let $A$ be an $f$-invariant $n$-point set, whose associated braid is denoted by $\beta\in B_n$ and its induced automorphism of $\mathbb{F}_n$ by $\phi$. The fixed points $p_1, p_2$, that correspond to the braids $u,v \in B_1(\mathbb{D}^2\setminus \{n\ \text{points}\})\cong \mathbb{F}_n$ are strong Nielsen equivalent, with respect to $A$, if and only if $\beta_u$ and $\beta_v$ are conjugate in $B_{n,1}$ via an element of $\mathbb{F}_n$. Note that $\beta_u$ and $\beta_v$ are the braids associated to $A\cup p_1$ and $A\cup p_2$ respectively. With this remark we justify how Theorem \ref{thm2} can be considered as a generalization of Theorem \ref{gu} for the case of periodic points.
\end{remark}

\begin{remark}
	The equivalence $o(x,f) \widesim{\text{\tiny SN}}_A o(y,f)$ if and only if 
	$\beta_{o_x}=\phi_{\beta_A}(c)\cdot \beta_{o_y}\cdot c^{-1}$, for $\beta_{o_x}, \beta_{o_y}, c\in B_m(\mathbb{D}^2\setminus \{n\ \text{points}\})$ and $\beta_A\in B_n$, of Theorem \ref{thm2}, can be seen as an equivalence between strong Nielsen classes, with respect to $A$, and Reidemeister classes, which gives an algebraic characterization of such strong Nielsen classes.
\end{remark}

\begin{remark}
	Let $x,y\in P_m(f)$. We would like to compare periodic Nielsen equivalence of $x,y$ with strong Nielsen equivalence of $x,y$ in terms of braids. We notice that the definition of strong Nielsen equivalence using paths takes into account the entire orbit of $\gamma:[0,1]\to \mathbb{D}^2$, while in the case of periodic Nielsen equivalence we take into account just its $m^{th}$ iterate. In terms of braids, and with respect to an $f$-invariant $n$-point set $A$, one can see from Theorem \ref{thm2} that in strong Nielsen equivalence we consider braids in $B_{n,m}$, while in periodic Nielsen equivalence we consider braids in $B_{n, 1}$. For periodic Nielsen equivalence we consider braids in $B_{n,1}$ since we can use Theorem \ref{gu}, seeing the periodic points as fixed points of the map $f^m$. As a result, from the $m$ stands of a braid in $B_{n,m}$ we get information about the motion of every point in the $m$-periodic orbit, while from the last stand of a braid in $B_{n,1}$ we get information only about the motion of the periodic point under $f^m$.

\end{remark}

\section{Braid forcing}\label{s4}
 
    The main reference for what follows is \cite{boju}. Let $f:(\mathbb{D}; \partial\mathbb{D}^2)\to (\mathbb{D}; \partial\mathbb{D}^2)$ be an orientation-preserving homeomorphism that fixes the boundary of $\mathbb{D}^2$ pointwise. The aim of this section is to show an application of Theorem \ref{thm2} to the problem of finding periodic orbits forced by a given one. In particular, given $f$ we show that the periodic orbits forced by a given one, obtained by the trace formula of Jiang--Zheng \cite{boju}, are all in different strong Nielsen classes. Let us first present important definitions and results we will use.
	
	With the aim to generalize the partial
	order that occurs in Sharkovski\u{\i}’s theorem about the periods of periodic orbits for maps
	of the line \cite{shark}, Boyland \cite{boy} and Matsuoka \cite{matsuoka} introduced the notion of braid types, defined in Section \ref{s2}, for characterizing the forcing relation in dimension 2. By forcing relation we mean that a given periodic orbit implies the existence of other orbits. We recall that in the case of $\mathbb{D}^2$ an $n$-braid type is in fact a conjugacy class in the braid group $B_n$.
	
	\begin{definition}
A braid $\beta'$ is an extension of $\beta$ if $\beta'$ is a union of $\beta$ and another braid $\gamma$. The braids $\beta,\ \gamma$ are disjoint but possibly linked.
	\end{definition}

As described in Section \ref{s2}, note that any braid $\beta'\in B_{n,m}$ is an extension of a braid $\beta\in B_n$ by a braid $\gamma\in B_{m}(\mathbb{D}^2\setminus\{n \ \text{points}\})$. 
We denote by $\beta_P\in B_n$ the braid that corresponds to the $f$-invariant subset $P\subset \text{Int}(\mathbb{D}^2)$ under the isotopy $f_t:f\simeq\text{id}_{\mathbb{D}^2}$. Moreover, by $[\beta]$ we denote the conjugacy class, in the group specified by the context, of a braid $\beta$.
	\begin{definition}
A braid $\beta$ forces a braid $\gamma$ if, for any $f$ and any isotopy $f_t:f\simeq\text{id}_{\mathbb{D}^2}$, the existence of an $f$-invariant set $P$ with $[\beta_P]=[\beta]$ guarantees the existence of an $f$-invariant set $Q$ with $[\beta_Q]=[\gamma]$. 
\end{definition}

\begin{definition}
An extension $\beta'$ is forced by $\beta$ if, for any $f$ and any isotopy $f_t:f\simeq\text{id}_{\mathbb{D}^2}$, the existence of an $f$-invariant set $P$ with $[\beta_P]=[\beta]$ guarantees the existence of an $f$-invariant set $Q\subset \mathbb{D}^2\setminus P$ with $[\beta_{P\cup Q}]=[\beta']$.
	\end{definition}

We are now ready to state the theorem by Jiang--Zheng \cite{boju}, where they deduce a trace formula for the computation of forced braid extensions.

\begin{theorem}[Jiang--Zheng, \cite{boju}]\label{trace}
Let $\beta' \in B_{n+m}$ be an extension of $\beta\in B_n$. Then $\beta'$ is forced by $\beta$ if and only if $\beta'$ is neither collapsible nor peripheral relative to $\beta$, and the conjugacy class $[\beta']$ has a non-zero coefficient in $\text{tr}_{B_{n+m}}\zeta_{n,m}(\beta)$.
\end{theorem}

By $\zeta_{n,m}$ they denote a matrix representation of $B_n$ over a free $\mathbb{Z}B_{n+m}$-module and the trace $\text{tr}_{B_{n+m}}$ is meant to take values in the free Abelian group generated by the conjugacy classes is $B_{n+m}$. Thus, to obtain forced extensions, that correspond to $(n+m)$-braids in $B_{n,m}$, of a given braid in $B_n$ we can apply the following steps. Compute the trace $\text{tr}_{B_{n+m}}$, merge conjugate terms in the sum by solving the conjugacy problem in $B_{n+m}$, then rule out certain unwanted terms, which correspond to collapsible and peripheral extensions, and finally the non-zero terms remaining after the cancellation are exactly the forced extensions, that correspond to $(n+m)$-braids, of a given $n$-braid. Roughly speaking, collapsible and peripheral extensions relative to $\beta$, are those braids which contain strands that may be merged or moved to infinity while keeping $\beta$ untouched. For more details we refer the reader to \cite{boju}.

In \cite{algorithm}, using Theorem \ref{trace} they provide an algorithm for solving the problem of finding periodic orbits linked to a given one, for orientation-preserving 2-disc homeomorphisms. We state their result in the following theorem.

\begin{theorem}[Liu,\cite{algorithm}]\label{liu}
Given a set $P$ of $n$ points in $\text{Int}(\mathbb{D}^2)$, a cyclic braid $\beta\in B_n$, and an integer $m>0$, there is an algorithm to decide a homeomorphism $f:\mathbb{D}^2\to \mathbb{D}^2$ fixing $\partial\mathbb{D}^2$ pointwise, $P$ as an $n$-periodic orbit of $f$ that induces the braid $\beta$, and $f$ has an $m$-periodic orbit with non-zero linking number about $P$.  
\end{theorem} 
 
We briefly describe the idea of the algorithm. Given a cyclic braid $\beta\in B_n$, due to the identification $B_n\cong \text{MCG}(\mathbb{D}^2, P; \partial\mathbb{D}^2)$, we obtain an orientation-preserving homeomorphism of the 2-disc with $P$ an $n$-periodic orbit of $f$, that corresponds to $\beta$. Then we can apply Theorem \ref{trace} to determine forced extensions $\beta'\in B_{n,m}$ of $\beta$ by computing the trace $\text{tr}_{B_{n+m}}\zeta_{n,m}(\beta)$. For more insights the reader is directed to \cite{algorithm}.
\\

\textbf{Application:} We recall that in Theorem \ref{thm2} we show that if two distinct extensions $\beta_x=\beta_{A}\cdot\beta_{o_x}\in B_{n,m}\subset B_{n+m}$ and $\beta_y=\beta_{A}\cdot\beta_{o_y}\in B_{n,m}\subset B_{n+m}$ of $\beta_A\in B_n$ are conjugate in $B_{n,m}\subset B_{n+m}$, by an element in a specific subset of $B_{n,m}$, then the periodic orbit that corresponds to $\beta_{o_x}$ is strong Nielsen equivalent to the periodic orbit that corresponds to $\beta_{o_y}$, with respect to the invariant set $A$. From Theorem \refeq{liu}, given a braid $\beta\in B_n$, that corresponds to an $n$-periodic point $p$ of $f$ we can determine forced extensions $\beta'=\beta\cdot\gamma\in B_{n,m}$, where $\gamma\in B_m$, by applying Theorem \ref{trace}. These extensions determine in their turn the forced periodic orbits of the points $q$, that correspond to the braids $\gamma\in B_m$. In order to compute the forced extensions one first has to compute the trace given in Theorem \ref{trace}, which is the sum of distinct conjugacy classes of extensions $\beta'$ of $\beta$. Thus, from Theorem \ref{thm2}, each representative of these conjugacy classes corresponds to a forced braid $\beta'$ of $\beta$, or equivalently, to a periodic orbit of a point $q$ forced by the periodic orbit of the periodic point $p$, which belongs to a different strong Nielsen class, with respect to the periodic orbit $p$. In other words, loosely speaking the trace $\text{tr}_{B_{n+m}}$ sums over the strong Nielsen classes of forced $m$-periodic orbits of $f$, with respect to $\beta$. Last but not least, for $m=1$, using Theorem \ref{trace}, we obtain forced periodic orbits of order 1, i.e. forced fixed points. As we mentioned above, the trace $\text{tr}_{B_{n+1}}$ sums over the strong Nielsen classes of fixed points, with respect to the given $\beta$. Thus, from Remark \ref{same}, it follows that the trace formula sums over the Nielsen classes of $f$, with respect to the given $\beta$. This should not be a surprising result, since the trace formula defined in \cite{boju} is actually a generalized version of the generalized Lefschetz number.

\section*{Acknowledgement}
The author is grateful to the University Center of
Excellence "Dynamics, Mathematical Analysis and Artificial Intelligence" at the
Nicolaus Copernicus University for the provided facilities, the great hospitality and the financial support. In addition, the author would like to thank Philip Boyland for his generous help and the useful discussions.

\bibliographystyle{plain}
\bibliography{bibsne}

\begin{thebibliography}{10}

\bibitem{asimov}
D.~Asimov and J.~Franks.
\newblock Unremovable closed orbits.
\newblock In {\em Geometric dynamics ({R}io de {J}aneiro, 1981)}, volume 1007
  of {\em Lecture Notes in Math.}, pages 22--29. Springer, Berlin, (1983).

\bibitem{birman}
J.~S. Birman.
\newblock {\em Braids, links, and mapping class groups}.
\newblock Annals of Mathematics Studies, No. 82. Princeton University Press,
  Princeton, N.J.; University of Tokyo Press, Tokyo, (1974).

\bibitem{boy}
P.~Boyland.
\newblock Topological methods in surface dynamics.
\newblock {\em Topology Appl.}, 58:223--298, (1994).

\bibitem{classification}
A.~Fathi, F.~Laudenbach, and V.~Poénaru.
\newblock {\em Travaux de {T}hurston sur les surfaces}.
\newblock (1979).

\bibitem{guaschi}
J.~Guaschi.
\newblock Nielsen theory, braids and fixed points of surface homeomorphisms.
\newblock {\em Topology Appl.}, 117(2):199--230, (2002).

\bibitem{guaschi2013survey}
J.~Guaschi and D.~Juan-Pineda.
\newblock A survey of surface braid groups and the lower algebraic {$K$}-theory
  of their group rings.
\newblock In {\em Handbook of group actions. {V}ol. {II}}, Adv. Lect. Math.,
  Vol. 32, pages 23--75. Int. Press, Somerville, MA, (2015).

\bibitem{hall}
T.~Hall.
\newblock Unremovable periodic orbits of homeomorphisms.
\newblock {\em Math. Proc. Cambridge Philos. Soc.}, 110:523--531, (1991).

\bibitem{jiang}
B.~Jiang.
\newblock {\em Lectures on {N}ielsen fixed point theory}, volume~14 of {\em
  Contemporary Mathematics}.
\newblock American Mathematical Society, Providence, RI, (1983).

\bibitem{boju}
B.~Jiang and H.~Zheng.
\newblock A trace formula for the forcing relation of braids.
\newblock {\em Topology}, 47:51--70, (2008).

\bibitem{kassel}
C.~Kassel and V.~Turaev.
\newblock {\em Braid groups}, volume 247 of {\em Graduate Texts in
  Mathematics}.
\newblock Springer, New York, (2008).

\bibitem{algorithm}
X.~Liu.
\newblock Braids and linked periodic orbits of disc homeomorphisms.
\newblock {\em Topology Appl.}, 342:108778, (2024).

\bibitem{matsuoka2}
T.~Matsuoka.
\newblock Braids of periodic points and a {$2$}-dimensional analogue of
  {S}harkovski\u{\i}'s ordering.
\newblock In {\em Dynamical systems and nonlinear oscillations ({K}yoto,
  1985)}, World Sci. Adv. Ser. Dynam. Systems, pages 58--72. World Sci.
  Publishing, Singapore, (1986).

\bibitem{matsuoka3}
T.~Matsuoka.
\newblock {Braid Type of the Fixed Point Set for Orientation-Preserving
  Embeddings on the Disk}.
\newblock {\em Tokyo Journal of Mathematics}, 18:457 -- 472, (1995).

\bibitem{matsuoka}
T.~Matsuoka.
\newblock Periodic points and braid theory.
\newblock In {\em Handbook of topological fixed point theory}, pages 171--216.
  Springer, Dordrecht, (2005).

\bibitem{shark}
A.~N. Sharkovski\u{\i}.
\newblock Co-existence of cycles of a continuous mapping of the line into
  itself.
\newblock {\em Ukrain. Mat. \v{Z}.}, 16:61--71, (1964).

\bibitem{thurston}
W.~P. Thurston.
\newblock On the geometry and dynamics of diffeomorphisms of surfaces.
\newblock {\em Bull. Amer. Math. Soc.}, 19:417--431, (1988).

\end{thebibliography}

\end{document}